\documentclass[a4paper,12pt]{elsarticle}
\usepackage{amssymb}
\usepackage{epsfig}
\usepackage{amsfonts}
\usepackage{amsmath}
\usepackage{euscript}
\usepackage{amscd}
\usepackage{amsthm}
\DeclareMathAlphabet{\mathpzc}{OT1}{pzc}{m}{it}

\newtheorem{thm}{Theorem}[section]
\newtheorem{lem}[thm]{Lemma}
\newtheorem{prop}[thm]{Proposition}

\newtheorem{cor}[thm]{Corollary}

\newdefinition{defn}[thm]{Definition}
\newdefinition{ex}[thm]{Example}
\newdefinition{rem}[thm]{Remark}
\newdefinition{note}{Note}

\newcommand\onto {- \hspace{-.300cm} - \hspace{-.300cm} \twoheadrightarrow}
\newcommand\m {\mathfrak{m}}
\newcommand\p {\mathfrak{p}}

\begin{document}
\begin{frontmatter}
\title{On codimension-one $\mathbb A^1$-fibration with retraction}
\author{Prosenjit Das}
\ead{prosenjit.das@gmail.com}
\author{Amartya K. Dutta}
\ead{amartya@isical.ac.in}
\address{Stat-Math Unit, Indian Statistical Institute, \\
203 B.T. Road, Kolkata 700 108, India}

\begin{abstract}
In this paper we prove some results on the sufficiency of codimension-one fibre conditions for a flat algebra with a retraction to be locally $\mathbb{A}^1$ or at least an $\mathbb{A}^1$-fibration.

\end{abstract}
\end{frontmatter}

\section[intro]{Introduction}
Throughout this paper, $R$ will denote a commutative ring with unity and $R^{[n]}$ a polynomial ring in $n$ variables over $R$. Let $A$ be an $R$-algebra. We shall use the notation $A = R^{[n]}$ to mean that $A$ is isomorphic, as an $R$-algebra, to a polynomial ring in $n$ variables over $R$.

\smallskip

For a prime ideal $P$ of $R$, $k(P)$ will denote the residue field $R_P/ P R_P$ and $A_P$ will denote the ring $S^{-1}A$, where $S = R \backslash P$. Thus $A \otimes_R k(P) = A_P/P A_P$.

\medskip

A finitely generated flat $R$-algebra $A$ is said to be an \textit{$\mathbb{A}^1$-fibration} over $R$ if $A \otimes_R k(P) = k(P)^{[1]}$ for all $P \in Spec \ R$.

\medskip

A \textit{retraction} $\Phi$ from an $R$-algebra $A$ to $R$ is a ring homomorphism $\Phi : A \longrightarrow R$ such that $\Phi|_R = 1_R$, i.e., it is an $R$-algebra homomorphism from $A$ to $R$. If a retraction exists, $R$ is said to be a \textit{retract} of $A$.

\medskip

Let $k$ be a field and $\bar{k}$ denote the algebraic closure of $k$. A $k$-algebra $B$ is said to be \textit{geometrically integral over $k$} if $B \otimes_k \bar{k}$ is an integral domain, and an \textit{$\mathbb{A}^1$-form over $k$} if $B \otimes_k \bar{k} = {\bar{k}}^{[1]}$.

\bigskip

The following result on $\mathbb{A}^1$-fibration was proved in (\cite{D_MOR}, 3.4, 3.5):

\begin{thm} \label{Thm_D_A1}
 Let $R$ be a Noetherian domain with quotient field $K$ and $A$ a faithfully flat finitely generated $R$-algebra such that $A \otimes_R K = K^{[1]}$ and $A \otimes_R k(P)$ is geometrically integral over $k(P)$ for each height one prime ideal $P$ of $R$. Under these hypotheses, we have the following results:
\begin{enumerate}
\item[\rm(i)] If $R$ is normal, then $A \cong Sym_R ( I )$ for an invertible ideal $I$ of $R$.
\item[\rm(ii)] If $R$ contains $\mathbb{Q}$, then $A$ is an $\mathbb{A}^1$-fibration over $R$.
\end{enumerate}
\end{thm}

A striking feature of this result is that conditions on merely the generic and codimension-one fibres imply that all fibres are $\mathbb{A}^1$. Analogous results were proved for subalgebras of polynomial algebras (\cite{BD_A1FIBSUBALG}, 3.10, 3.12) without the hypothesis ``$A$ is finitely generated over $R$''. In this paper we investigate whether the condition ``$A$ is finitely generated'' in Theorem \ref{Thm_D_A1} can be replaced by a weaker hypothesis like ``$A$ is Noetherian'' when the $R$ algebra $A$ is known to have a retraction to $R$. Recently, in \cite{LOC_A1_CODIM1}, Bhatwadekar-Dutta-Onoda have shown the following:

\begin{thm} \label{BDO}
Let $R$ be a Noetherian normal domain with field of fractions $K$ and $A$ a Noetherian flat $R$-algebra such that $A_P = {R_P}^{[1]}$ for each prime ideal $P$ of $R$ of height one. Suppose that there
exists a retraction $\Phi : A \onto R$. Then $A \cong Sym_R ( I )$ for an invertible ideal $I$ in $R$.
\end{thm}

\medskip

The above theorem occurs in \cite{LOC_A1_CODIM1} as a consequence of a general structure theorem for any faithfully flat algebra over a Noetherian normal domain $R$ which is locally $\mathbb{A}^1$ in codimension-one. The statements and proofs in \cite{LOC_A1_CODIM1} are quite technical. In this paper, we will first prove (see Theorem \ref{Th2}) an analogue of Theorem \ref{Thm_D_A1} (i). Our approach, which is more in the spirit of the proof in (\cite{D_MOR}, 3.4), will provide a short and direct proof of Theorem \ref{BDO}. Next we will prove the following analogue of Theorem \ref{Thm_D_A1} (ii) (see Theorem \ref{Th4}):
\bigskip

\noindent 
{\bf Theorem A.}
{\it
Let $ \mathbb{Q} \hookrightarrow R$ be a Noetherian domain with quotient field $K$ and $A$ be a Noetherian flat $R$-algebra with a retraction $\Phi : A \onto R$ such that

\begin{enumerate}
 \item [\rm(1)]$A \otimes_R K = K^{[1]}$.
 \item [\rm(2)]$A \otimes_R k(P)$ is an integral domain for each height one prime ideal $P$ of $R$.
\end{enumerate}

Then $A$ is an $\mathbb{A}^1$-fibration over $R$. Thus, if $R$ is seminormal, then $A \cong Sym_R (I)$ for some invertible ideal $I$ of $R$.
}

\medskip

As an intermediate step, we shall prove the following result (see Proposition \ref{Th3}) which gives a generalisation of Theorem \ref{BDO} over an arbitrary Noetherian domain:

\medskip

\noindent 
{\bf Proposition A.}
{\it
Let $R$ be a Noetherian domain with quotient field $K$ and $A$ be a Noetherian flat $R$-algebra with a retraction $\Phi : A \onto R$ such that

\begin{enumerate}
 \item [\rm (1)]$A \otimes_R K = K^{[1]}$.
 \item [\rm (2)]$A \otimes_R k(P)$ is geometrically integral over $k(P)$ for each height one prime ideal $P$ of $R$.
\end{enumerate}

Then $A$ is finitely generated over $R$ and there exists a finite birational extension $R'$ of $R$ and an invertible ideal $I$ of $R'$ such that $A \otimes_R R' \cong Sym_{R'} ( I )$.
}

\medskip

In fact, in our results, the hypothesis ``$A$ is Noetherian'' can be replaced by ``$Ker \ \Phi$ is finitely generated''.

\section[RS]{A version of Russell-Sathaye criterion for an algebra to be a polynomial algebra}

In this section we present a version of Russell-Sathaye criterion (\cite{RS_FIND}, 2.3.1) for an algebra to be a polynomial algebra. Our version is an extension of the version given by Dutta-Onoda (\cite{DO_CODIM1}, 2.4) and suitable for algebras which are known to have retractions to the base ring. For convenience, we first record a few preliminary results. The first two results are easy.

\begin{lem} \label{L0a}
Let $B \subset A$ be integral domains. 
Suppose that there exists a non-zero element $p$ in $B$ such that $B[1/p] = A[1/p]$ and $pA \cap B = pB$. 
Then $B = A$. 
\end{lem}

\begin{lem} \label{L0b}
 Let $C$ be a $D$-algebra such that $D$ is a retract of $C$. Then the following hold:

\begin{enumerate}
 \item [\rm (I)]$pC \cap D = pD$ for all $p \in D$.
 \item [\rm (II)]If $D \subset C$ are domains, then $D$ is algebraically closed in $C$. 
\end{enumerate}
\end{lem}

\begin{lem}\label{L1}
Let $R$ be a ring and $A$ be an $R$-algebra with a generating set $S= \{x_i: i \in \Lambda \}$ where $\Lambda$ is some indexing set. Suppose that there is a retraction $\Phi : A \onto R$. Then $Ker \ \Phi = (\{ x_i-r_i: i \in \Lambda \})A$ where $r_i = \Phi(x_i)$ for each $i \in \Lambda$.
\end{lem}

\begin{proof}
Let $\widetilde{S}=\{ x_i-r_i: i \in \Lambda \}$ and $I$ be the ideal of $A$ generated by $\widetilde{S}$. Note that $R[S] = R[\widetilde{S}]$. It is easy to see that $A = R \oplus Ker \ \Phi = R \oplus I$. Since $I \subseteq Ker \ \Phi$, it follows that $Ker \ \Phi = I$.
\end{proof}

\begin{lem} \label{L2}
Let $R \subset A$ be integral domains and $\Phi: A \onto R$ be a retraction with finitely generated kernel. Suppose that there exists an element $p$ which is a non-zero non-unit in $R$ such that $A[1/p] = R[1/p]^{[1]}$. Then there exists $x \in Ker \ \Phi$ such that $x \notin p A$ and $A[1/p] = R[1/p][x]$.
\end{lem}

\begin{proof}
Suppose, if possible, that $x \in pA$ for every $x \in Ker \ \Phi$ for which $A[1/p] = R[1/p][x]$.

\medskip

Let $Ker \ \Phi = (a_1, a_2, \dots, a_m)A$. Choose $x_0 \in Ker \ \Phi$ such that $A[1/p] = R[1/p][x_0]$. Note that $\Phi$ extends to a retraction $\Phi_p: A[1/p] \onto R[1/p]$ with kernel $x_0(A[1/p])$. By our assumption, $x_0 = p x_1$ for some $x_1 \in A$. Obviously, $x_1 \in Ker \ \Phi$ and $A[1/p] = R[1/p][x_1]$ and hence $x_1 \in p A$. Arguing in a similar manner, we get $x_2 \in Ker \ \Phi$ such that $x_1 = p x_2$, $A[1/p] = R[1/p][x_2]$ and $x_2 \in p A$. Continuing this process we get a sequence $\{ x_n \}_{n \ge 0}$ such that $x_n \in Ker \ \Phi$, $A[1/p] = R[1/p][x_n]$ and $x_n = p x_{n+1}$. Thus $x_0 = p^n x_n$ for all $n \ge 1$.

\medskip

Note that $(x_0, x_1, \dots, x_n, \dots)A \subseteq (a_1, a_2, \dots, a_m)A$. But since $a_i \in A \subset A[1/p] = R[1/p][x_0]$, there exist $n_i \in \mathbb{N}$ and $\alpha_{ij} \in R[1/p]$ such that $a_i = \underset{j=0} {\overset{n_i} {\sum}} \alpha_{ij} {x_0}^{j}$. Choose $N \in \mathbb{N}$ such that $\alpha_{ij} p^{jN} \in R$ for all $i, j$ and set $\lambda_{ij} := \alpha_{ij} p^{jN}$. Now since $ x_0, a_i \in Ker \ \Phi_p$, we have $\alpha_{i0} =0$ for all $i$ and hence $a_i = \underset{j=1} {\overset{n_i} {\sum}} \alpha_{ij} {x_0}^{j}$. Thus $a_i = \underset{j=1} {\overset{n_i} {\sum}} \lambda_{ij} {x_N}^{j} \in x_N R[x_N] \subseteq x_N A$ for all $i$, $1 \le i \le m$. So, we have $Ker \ \Phi=(a_1, a_2, \dots, a_m)A = x_NA$. Now $x_{N+1} \in Ker \ \Phi = x_NA$, which implies that $x_{N+1} = \alpha x_N$ for some $\alpha \in A$. Since $x_N = p x_{N+1}$, it follows that $\alpha p =1$, which is a contradiction to the fact that $p$ is not a unit in $A$.

\medskip

Thus there exists $x \in Ker \ \Phi$ such that $x \notin p A$ and $A[1/p] = R[1/p][x]$.
\end{proof}

\medskip

We now present a version of Russell-Sathaye criterion 
when there exists a retraction.

\begin{prop} \label{Th1}
Let $R \subset A$ be integral domains such that there exists a retraction $\Phi: A \onto R$. Suppose that there exists a prime $p$ in $R$ such that

\begin{enumerate}
 \item [\rm (1)]$p$ is a prime in $A$.
 \item [\rm (2)]$A[1/p]  = R[1/p]^{[1]}$.
\end{enumerate}

Then $pA \cap R = pR$, $R/pR$ is algebraically closed in $A/pA$ and there exists an increasing chain $A_0 \subseteq A_1 \subseteq A_2 \ ...  \subseteq A_n \subseteq  ... $ of subrings of A and a sequence of elements $\{ x_n \}_{n \ge 0}$ in $Ker \ \Phi$ with $x_0 A \subseteq x_1 A \subseteq \dots \subseteq x_n A \subseteq \dots$ such that

\begin{enumerate}
 \item [\rm (a)]$A_n = R[x_n] = R^{[1]}$ for all $n \ge 0$.
 \item [\rm (b)]$A[1/p]=A_n[1/p]$ for all $n \ge 0$.
 \item [\rm (c)]$pA \cap A_n \subseteq pA_{n+1}$ for all $n \ge 0$.
 \item [\rm (d)]$A= \underset{n \ge 0}{\cup}A_n = R[x_1, x_2, \dots, x_n, \dots]$.
 \item [\rm (e)]$Ker \ \Phi = (x_0, x_1, x_2, \dots, x_n, \dots)A$.
\end{enumerate}

Moreover the following are equivalent:

\begin{enumerate}
 \item [\rm (i)] $Ker \ \Phi$ is finitely generated.
 \item [\rm (ii)] $Ker \ \Phi = x_N A$ for some $N \ge 0$.
 \item [\rm (iii)] $A$ is finitely generated over $R$.
 \item [\rm (iv)] $A = R[x_N]$ for some $N \ge 0$.
 \item [\rm (v)] There exists $x \in Ker \ \Phi \backslash pA $ 
 such that $A = R[x] = R^{[1]}$.
\end{enumerate}

The conditions (i)--(v) will be satisfied if 
$\underset{n \ge 0}{\cap} p^n A = (0)$.
\end{prop}

\begin{proof}
$pA \cap R = pR$ by Lemma \ref{L0b}. Since $\Phi$ induces a 
retraction $ \Phi_p : A/pA \onto R/pR$, $R/pR$ is algebraically closed 
in $A/pA$ by Lemma \ref{L0b}.

\medskip

By condition (2), there exists ${x'}_0 \in A$ such that 
$A[1/p]=R[1/p][{x'}_0]$. Let $x_0 = {x'}_0 - \Phi({x'}_0)$. 
Then $x_0 \in Ker \ \Phi$ and $A[1/p]=R[1/p][x_0]={R[1/p]}^{[1]}$. 
Set $A_0 := R[x_0] (=R^{[1]})$. 
Then $A_0 \subseteq A$ and $A[1/p]=A_0[1/p]=R[1/p][x_0]$.

Now suppose that we have obtained elements 
$x_0, x_1, \dots , x_n \in Ker \ \Phi$ such that setting 
$A_m := R[x_m] (=R^{[1]})$ for all $m$, $0 \le m \le n$, 
we have $A_0 \subseteq A_1 \subseteq A_2 \ \dots \subseteq A_n \subseteq A$ 
and $A_m[1/p] = A[1/p]$; $0 \le m \le n$.

\medskip

We now describe our choice of $x_{n+1}$:

\smallskip
Let $\overline{x_n}$ denote the image of $x_n$ in $A/pA$. 
We consider separately the two possibilities: 
\begin{enumerate}
 \item [\rm (I)]$\overline{x_n}$ is transcendental over $R/pR$.
 \item [\rm (II)]$\overline{x_n}$ is algebraic over $R/pR$.
\end{enumerate}

Case I : $\overline{x_n}$ is transcendental over $R/pR$. 
In this case the map $A_n/pA _n(=R[x_n]/pR[x_n]) \longrightarrow A/pA$ 
is injective, i.e., $pA_n = pA \cap A_n$. 
Since $A_n[1/p] = A[1/p]$, we get $A_n = A$ by Lemma \ref{L0a}. 
Now we set $x_{n+1} := x_n$ and $A_{n+1} := R[x_{n+1}] (=A_n = A)$.

\medskip

Case II: $\overline{x_n}$ is algebraic over $R/pR$. 
Since $R/pR$ is algebraically closed in $A/pA$, 
we see that $\overline{x_n} \in R/pR$. 
Thus $x_n = pu_n + c_n$ for some $u_n \in A$ and $c_n \in R$. 
Applying $\Phi$, we get $0=\Phi(x_n)=p\Phi(u_n) + c_n$ 
showing that $c_n \in pR$ and hence $x_n \in pA$. 
Set $x_{n+1} := x_n/p (\in A)$. Clearly $x_{n+1} \in Ker \ \Phi$. 
Now setting $A_{n+1} := R[x_{n+1}] (=R^{[1]})$, 
we see that $A_0 \subseteq A_1 \subseteq A_2 \ \dots \subseteq A_n 
\subseteq A_{n+1} \subseteq A$ and $A_{n+1}[1/p] = A_n[1/p] = A[1/p]$.

\medskip

Thus we set $x_{n+1} := x_n$ or $x_{n+1} := x_n/p$ depending on 
whether the image of $x_n$ in $A/pA$ is transcendental or 
algebraic over $R/pR$. By construction, conditions (a) and (b) hold. 
We now verify (c).

\medskip

If $x_n = x_{n+1}$, i.e., $A_{n+1} = A_n = A$, then 
$pA \cap A_n = pA = pA_{n+1}$. Now consider the case 
$x_n = px_{n+1} \in pA_{n+1}$. Let $a \in pA \cap A_n$. 
Then $a = r_0 + r_1 (px_{n+1}) + \dots + r_l (px_{n+1})^l$ 
for some $l \ge 0$ and $r_0, r_1, \dots, r_l \in R$. 
Then $r_0 \in pA \cap R = pR \subset pA_{n+1}$. Therefore, $a \in pA_{n+1}$. 
Thus $pA \cap A_n \subseteq A_{n+1}$.

\medskip

We now prove (d). Let $B =\underset{n \ge 0}{\cup} A_n$. 
Obviously, $B \subseteq A$ and $B[1/p]=A[1/p]$. 
Hence, by Lemma \ref{L0a}, it is enough to show that $pA \cap B =pB$.

\smallskip

Clearly, $pB \subseteq pA \cap B$. Now let $y \in pA \cap B$. 
Then there exists $i \ge 0$ such that 
$y \in pA \cap A_i \subseteq pA_{i+1} \subseteq pB$. Thus $pA \cap B =pB$.

\medskip

$(e)$ follows from Lemma \ref{L1}.

\medskip

(i) $\Longrightarrow$ (v) follows from Lemma \ref{L2}. 
Our construction shows that (iii) $\Longrightarrow$ (iv). 
The implications (v) $\Longrightarrow$ 
(iii) and (iv) $\Longrightarrow$ (ii) $\Longrightarrow$ (i) are easy.

\medskip

Note that our construction shows that the sequence $\{x_n\}_{n \ge 0}$ 
is eventually a constant sequence (i.e., there exists $N \ge 0$ such that 
$x_{N+r} = x_N$ for all $r \ge 0$) if and only if there exists 
$N \ge 0$ such that the image of $x_N$ in $A/pA$ is transcendental 
over $R/pR$. It is easy to see that each of the conditions 
(i)--(v) is equivalent to the above condition.

\medskip

If the image of $x_m$ in $A/pA$ is algebraic over $R/pR$ for 
$1 \le m \le n$, then $x_n = p^n x_0 \in p^n A$. 
Therefore, if $\underset{n \ge 0}{\cap} p^n A = (0)$, 
then the sequence $\{x_n\}_{n \ge 0}$ must be eventually constant 
and hence the conditions (i)--(v) hold.
\end{proof}

Proposition \ref{Th1} shows that we can extend the 
Dutta-Onoda version (\cite{DO_CODIM1}, 2.4) of Russell-Sathaye 
criterion for $A$ to be $R^{[1]}$ as follows:

\begin{cor} \label{Cor_RS_DO-ver}
 Let $R \subset A$ be integral domains. 
 Suppose that there exists a prime $p$ in $R$ such that

\begin{enumerate}
 \item [\rm (1)]$p$ is a prime in $A$.
 \item [\rm (2)]$pA \cap R = pR$.
 \item [\rm (3)]$A[1/p] = R[1/p]^{[1]}$.
 \item [\rm (4)]$R/pR$ is algebraically closed in $A/pA$.
\end{enumerate}
Then the following are equivalent:
\begin{enumerate}
 \item [\rm (i)]$A$ is finitely generated over $R$.
 \item [\rm (ii)]$A$ has a retraction to $R$ with finitely generated kernel.
 \item [\rm (iii)]${trdeg}_{R/pR} (A/pA) > 0$.
 \item [\rm (iv)]$A = R^{[1]}$.
\end{enumerate}
\end{cor}

\begin{proof}
 Follows from (\cite{DO_CODIM1}, 2.4) and Proposition \ref{Th1}.
\end{proof}

By repeated application of Proposition \ref{Th1} we get the following:

\begin{cor} \label{Cor1_Th1}
Let $R \subset A$ be integral domains with a retraction $\Phi: A \onto R$. 
Suppose that there exist primes $p_1, p_2, \dots, p_n$ in $R$ such that

\begin{enumerate}
 \item [\rm (1)]$Ker \ \Phi$ is finitely generated.
 \item [\rm (2)]$p_1, p_2, \dots, p_n$ are primes in $A$.
 \item [\rm (3)]$A[\frac{1}{p_1 p_2 \dots p_n }]  = 
 R[\frac{1}{p_1 p_2 \dots p_n }]^{[1]}$.
\end{enumerate}

Then there exists $x \in Ker \ \Phi$ such that $A = R[x] = R^{[1]}$.
\end{cor}

\section{Codimension-one $\mathbb A^1$-fibration with retraction}
In this section we will prove our main theorems (Theorems \ref{Th2} 
and \ref{Th4}) and auxiliary results (Propositions \ref{Th0} and \ref{Th3}).

\medskip

We first state a few preliminary results. 
The first result occurs in (\cite{BD_A1FIBSUBALG}, 3.4).

\begin{lem} \label{B-D_lem_1}
 Let $R$ be a Noetherian ring and $R_1$ a ring containing $R$ which is 
 finitely generated as an $R$-module. 
 If $A$ is a flat $R$-algebra such that $A \otimes_R R_1$ 
 is a finitely generated $R_1$-algebra, then 
 $A$ is a finitely generated $R$-algebra.
\end{lem}

The following result follows from (\cite{BD_A1FIBSUBALG}, 3.3 and 3.5).

\begin{lem} \label{B-D_lem_2}
 Let $R$ be a Noetherian ring and $A$ a flat $R$-algebra such that, 
 for every minimal prime ideal $P$ of $R$,
 $PA$ is a prime ideal of $A$,  $PA \cap R=P$ and
 $A/PA$ is finitely generated over $R/P$.  
 Then $A$ is finitely generated over $R$.
\end{lem}

The next result is easy to prove.

\begin{lem} \label{L4}
 Let $R$ be a ring and $A$ an $R$-algebra. 
 If $R'$ is a faithfully flat algebra over $R$ such that 
 $A \otimes_R R'$ is finitely generated over $R'$, 
 then $A$ is finitely generated over $R$.
\end{lem}

We now quote a theorem on finite generation due to N. Onoda 
(\cite{O_Subring}, 2.20).

\begin{thm} \label{O_f.g.}
Let $R$ be a Noetherian domain and let $A$ be an integral domain 
containing $R$ such that

\begin{enumerate}
\item [\rm (1)]There exists a non zero element $t \in A$ for which 
$A[1/t]$ is a finitely generated $R$-algebra.
\item [\rm (2)]$A_{\m}$ is a finitely generated $R_{\m}$-algebra 
for each maximal ideal $\m$ of $R$.
\end{enumerate}

Then $A$ is a finitely generated $R$-algebra.
\end{thm}

The results on $\mathbb{A}^1$-fibrations in 
(\cite{BD_A1FIBSUBALG}, \cite{D_MOR}, \cite{DO_CODIM1}) crucially 
involve certain patching techniques. 
We state below one such ``patching lemma'' 
(\cite{DO_CODIM1}, 3.2).

\begin{lem} \label{DO_L2}
 Let $R \subset A$ be integral domains with 
 $A$ being faithfully flat over $R$. Suppose that there exists a 
 non-zero element $t \in R$ such that

\begin{enumerate}
 \item [\rm(1)]$A[1/t] = R[1/t]^{[1]}$.
 \item [\rm(2)]$S^{-1}A = (S^{-1} R)^{[1]}$, where 
 $S = \{ r\in R| \ r$ is not a zero-divisor in $R/tR\}$.

\end{enumerate}

Then there exists an invertible ideal $I$ in $R$ such that $A \cong Sym_R (I)$.
\end{lem}

We now observe a property of algebras with retractions.

\begin{lem} \label{L3}
Let $R$ be an integral domain with quotient field $K$ 
and $A$ be an integral domain containing $R$ with a 
retraction $\Phi: A \onto R$ such that

\begin{enumerate}
\item [\rm (1)] $Ker \ \Phi$ is finitely generated.
\item [\rm (2)] $A \otimes_R K = K^{[1]}$.
\end{enumerate}

Then there exists $t \in R$ and $F \in Ker \ \Phi$ such that 
$A[1/t] = R[1/t][F]$.
\end{lem}

\begin{proof}
 Let $S = R \backslash \{ 0 \}$. By (2), $S^{-1}A = K^{[1]}$. 
 Since $A$ has a retraction $\Phi$, it is easy to see that 
 there exists $F \in Ker \ \Phi$ such that $S^{-1}A = K[F] 
 (=K^{[1]})$ and hence $F (S^{-1}A) = (Ker \ \Phi) S^{-1}A$. 
 Therefore, by (1), there exists $t \in S$ such that 
 $F A[1/t] = (Ker \ \Phi) A[1/t]$. Thus 
 $F A[1/t]$ is the kernel of the induced retraction 
 $\Phi_t : A[1/t] \onto R[1/t]$. Hence we have
\begin{eqnarray*}
A[1/t] &=& R[1/t] \oplus F A[1/t]\\
&=& R[1/t] \oplus F R[1/t] \oplus F^2 A[1/t]\\ 
&\ldots& \\
&=& R[1/t] \oplus F  R[1/t] \oplus F^2  R[1/t] \oplus \dots \oplus 
F^n  R[1/t] \oplus F^{n+1}  A[1/t] \ \ \forall n \in \mathbb{N}.
\end{eqnarray*}
As $S^{-1}A = \underset{n \ge 0} {\overset{}{\oplus}} K F^n$, 
it follows that $A[1/t] = R[1/t][F]$.
\end{proof}

\begin{rem}
In Lemma \ref{L3} if we assume that $Ker \ \Phi$ is principal, 
say, $Ker \ \Phi = (G)$, then $A = R[G]$.
\end{rem}

We now deduce a local-global result. Our approach gives a 
simpler proof of Theorem \ref{BDO} which is obtained in 
\cite{LOC_A1_CODIM1} as a consequence of a 
highly technical structure theorem.

\begin{prop} \label{Th0}
Let $R$ be either a Noetherian domain or a Krull domain 
with quotient field $K$ and $A$ a flat $R$-algebra with 
a retraction $\Phi : A \onto R$ such that 

\begin{enumerate}
 \item [\rm(1)]$Ker \ \Phi$ is finitely generated.
 \item [\rm(2)]$A_P = {R_P}^{[1]}$ for every prime ideal $P$ of $R$ 
 satisfying $depth \ (R_P) =1$.
\end{enumerate}

Then there exists an invertible ideal $I$ of $R$ such that 
$A \cong Sym_R (I)$.
\end{prop}

\begin{proof}
The case $dim \ R = 0$ is trivial. So we assume that 
$dim \ R \ge 1$. Note that $A$ is a faithfully flat $R$-algebra 
and an integral domain. By Lemma \ref{L3}, $A[1/t]=R[1/t][F]$. 
If $t$ is a unit in $R$, then $A=R^{[1]}$ and we would be through. 
So we assume that $t$ is a non-unit in $R$.

\smallskip

Let $P_1, P_2, ..., P_s$ be the associated prime ideals of $tR$. 
Let $S = R \backslash (\underset{i=1} {\overset{s}{\cup}}{P_i}) = 
\{ r\in R| \ r$ is not a zero-divisor in $R/tR\} $. 
By (2), for each maximal ideal $\m$ of $S^{-1}R$, 
${(S^{-1}A)}_{\m}={{(S^{-1}R)}_{\m}}^{[1]}$ and 
hence $S^{-1} A = {(S^{-1}R)}^{[1]}$, $S^{-1}R$ being a semilocal domain. 
Hence, by Lemma \ref{DO_L2},  
$A \cong Sym_R ( I )$ for some invertible ideal $I$ of $R$.
\end{proof}

We now prove Theorem A for the case $R$ is a Krull domain.

\begin{thm} \label{Th2}
Let $R$ be a Krull domain with quotient field $K$ and $A$ a flat 
$R$-algebra with a retraction $\Phi : A \onto R$ such that

\begin{enumerate}
 \item [\rm(1)]Ker $\Phi$ is finitely generated.
 \item [\rm(2)]$A \otimes_R K = K^{[1]}$.
 \item [\rm(3)]$A \otimes_R k(P)$ is an integral domain for each 
 height one prime ideal $P$ of $R$.
\end{enumerate}

Then there exists an invertible ideal $I$ of $R$ such that 
$A \cong Sym_R ( I )$.
\end{thm}

\begin{proof}
Let $P$ be a prime ideal in $R$ for which $depth \ (R_P) (= ht \ P ) = 1$. 
Then $R_P$ is a DVR. Let $\pi_P$ be the uniformising parameter of 
$R_P$. Note that the retraction $\Phi : A \onto R$ induces a 
retraction $\Phi_P : A_P \onto R_P$ with finitely generated kernel, 
condition (2) ensures that $A_P[1/\pi_P]=R_P[1/\pi_P]^{[1]}=K^{[1]}$, 
and condition (3) ensures that $\pi_P$ is a prime in $A_P$. 
Hence, by Corollary \ref{Cor1_Th1}, $A_P = {R_P}^{[1]}$. 
Therefore, by Proposition \ref{Th0}, $A \cong Sym_R (I)$ for some 
invertible ideal $I$ of $R$.
\end{proof}

As an immediate consequence we get the following variant of a 
L\" uroth-type result over UFD (see \cite{RS_FIND}, 3.4):

\begin{cor} \label{Cor1_Th2}
Let $R$ be a UFD with quotient field $K$ and $A$ a flat $R$-algebra with a 
retraction $\Phi: A \onto R$ such that

\begin{enumerate}
 \item [\rm(1)]$Ker \ \Phi$ is finitely generated.
 \item [\rm(2)]$A \otimes_R K  = K^{[1]}$.
 \item [\rm(3)]$A \otimes_R k(P)$ is an integral domain 
 for each height one prime ideal $P$ of $R$.
\end{enumerate}
Then there exists $x \in Ker \ \Phi$ such that $A = R[x] = R^{[1]}$.
\end{cor}

\bigskip

We now prove Proposition A.

\begin{prop} \label{Th3}
Let $R$ be a Noetherian ring and $A$ be a flat $R$-algebra with a 
retraction $\Phi : A \onto R$ such that

\begin{enumerate}
 \item [\rm(1)] Ker $\Phi$ is finitely generated.
 \item [\rm(2)] $A \otimes_R k(P) = k(P)^{[1]}$ for each minimal prime ideal
  $P$ of $R$.
 \item [\rm(3)] $A \otimes_R k(P)$ is geometrically integral over $k(P)$
  for each height one prime ideal $P$ of $R$.
\end{enumerate}
Then:
\begin{enumerate}
 \item [\rm(I)] $A \otimes_R k(P)$ is an $\mathbb{A}^1$-form over $k(P)$ 
 for each prime ideal $P$ of $R$.
 \item [\rm(II)] $A$ is finitely generated over $R$.
 \item [\rm(III)] If $R$ is an integral domain, then there exists a 
 finite birational extension $R'$ of $R$ and an invertible ideal $I$ of $R'$ 
 such that $A \otimes_R R' \cong Sym_{R'} ( I )$.
\end{enumerate}
\end{prop}

\begin{proof}
(I):
Note that, for any prime ideal $P$ of $R$, 
$A \otimes_R k(P) = A_P \otimes_{R_P} k(P)$. So, 
to prove the fibre condition (I), we replace $R$ by $R_P$ 
(and $A$ by $A_P$) and assume that $R$ is a local ring with maximal ideal 
$\m$. We prove that $A \otimes_R k(\m)$ is an $\mathbb{A}^1$-form over $k(\m)$ 
by induction on height $\m$, i.e., dim $R$.

\medskip

Case : \underline{dim $R$ =0}.

Trivial.

\medskip

Case : \underline{dim $R$ =1}.

Replacing $R$ by $R/P_0$ for some minimal prime ideal $P_0$, 
we may assume that $R$ is a Noetherian one-dimensional local 
integral domain with quotient field $K$. Note that condition 
(3) implies that $A \otimes_R k(\m)$ is geometrically integral over $k(\m)$.

\medskip

Let $\widetilde{R}$ be the normalisation of $R$ and let 
$\widetilde{A} = A \otimes_{R} \widetilde{R}$. 
Then, by Krull-Akizuki theorem, $\widetilde{R}$ is a Dedekind domain 
(\cite{Mat_RING}, p 85); and since $R$ is local, 
$\widetilde{R}$ is semilocal and hence a PID. 
Let $\widetilde{\m}_1, \widetilde{\m}_2, \dots, \widetilde{\m}_r$ be 
the maximal ideals of $\widetilde{R}$. Again, 
by Krull-Akizuki theorem (\cite{Mat_RING}, p 85), 
$k(\widetilde{\m}_i)$ is a finite algebraic extension of $k(\m)$. 
Clearly, the retraction $\Phi : A \onto R$ gives rise to a 
retraction $\widetilde{\Phi}: \widetilde{A} \onto \widetilde{R}$. 
From the split exact sequence 
$0 \longrightarrow Ker \ \Phi \longrightarrow A 
\longrightarrow R \longrightarrow 0$, 
it follows that $Ker \ \widetilde{\Phi} = Ker \ \Phi \otimes_R 
\widetilde{R} = Ker \ \Phi \otimes_A \widetilde{A} = 
(Ker \ \Phi) \widetilde{A}$ and hence 
$Ker \ \widetilde{\Phi}$ is finitely generated. 

\medskip

Thus, from (1), (2) and (3), we have:

\smallskip

(i) $Ker \ \widetilde{\Phi}$ is finitely generated.

(ii) $\widetilde{A} \otimes_{\widetilde{R}} K = K^{[1]}$.

(iii) $\widetilde{A} \otimes_{\widetilde{R}} k(\widetilde{\m}_i)$ is 
geometrically integral over $k(\widetilde{\m}_i)$ 
for every maximal ideal $\widetilde{\m}_i$ of $\widetilde{R}$.

\smallskip

Hence, by Corollary \ref{Cor1_Th2}, $\widetilde{A} = \widetilde{R}^{[1]}$. 
In particular, $\widetilde{A} \otimes_{\widetilde{R}} 
k(\widetilde{\m}_i)= k(\widetilde{\m}_i)^{[1]}$ 
for each maximal ideal $\widetilde{\m}_i$ of $\widetilde{R}$. 
This shows that $A \otimes_R k(\m)$ is an $\mathbb{A}^1$-form over $k(\m)$.

\medskip

Case : \underline{dim $R \ge 2$}.

By induction hypothesis we have that 
$A \otimes_R k(P)$ is an $\mathbb{A}^1$-form for every non-maximal 
prime ideal $P$ of $R$. 
Let $\widehat{R}$ denote the completion of $R$ and let 
$\widehat{A} = A \otimes_R \widehat{R}$. 
Then $\widehat{R}$ is a complete local ring with 
maximal ideal $\widehat{\m}$ and $\widehat{R}/\widehat{\m} \cong R/\m$. 
Since $R$ is Noetherian, $\widehat{R}$ is Noetherian and 
faithfully flat over $R$ and hence $\widehat{A}$ is faithfully flat 
over both $A$ and $\widehat{R}$. 
The retraction $\Phi : A \onto R$ gives rise to a retraction 
$\widehat{\Phi}: \widehat{A} \onto \widehat{R}$. 
Note that $Ker \ \widehat{\Phi} = (Ker \ \Phi) \widehat{A}$ is 
finitely generated. Now, for any non-maximal prime ideal 
$\widehat{P}$ of $\widehat{R}$, $\widehat{P} \cap R \neq \m$ and 
hence $\widehat{A} \otimes_{\widehat{R}} k(\widehat{P})$ is an 
$\mathbb{A}^1$-form over $k(\widehat{P})$.

\medskip

Replacing $R$ by $\widehat{R}$, we may assume $R$ to be a 
complete local Noetherian ring. Further, replacing $R$ by $R/P_0$, 
where $P_0$ is a minimal prime ideal of $R$, 
we may assume $R$ to be a complete, local, Noetherian domain 
with maximal ideal $\m$ and quotient field $K$ such that 

\smallskip

(a) $A \otimes_R K = K^{[1]}$.

(b) $A \otimes_R k(P)$ is an $\mathbb{A}^1$-form over $k(P)$ 
for each non-maximal prime ideal $P$ of $R$.

\medskip

Let $\widetilde{R}$ be the normalisation of $R$. Since $R$ is a 
complete local ring, $\widetilde{R}$ is a finite $R$-module 
(\cite{Mat_RING}, p 263) and hence is a Noetherian normal local domain. 
Let $\widetilde{A} = A \otimes_R \widetilde{R}$. 
As before, the retraction $\Phi: A \onto R$ induces a 
retraction $\widetilde{\Phi}: \widetilde{A} \onto \widetilde{R}$ 
with finitely generated kernel $(Ker \ \Phi) \widetilde{A}$. 
Now we have the following:

\medskip

$\widetilde{R}$ is a Noetherian normal local domain with 
quotient field $K$ and $\widetilde{A}$ is a faithfully flat 
$\widetilde{R}$-algebra such that

\smallskip

($1^{\prime}$) There is a retraction 
$\widetilde{\Phi} : \widetilde{A} \onto \widetilde{R}$ with 
finitely generated kernel.

($2^{\prime}$) $\widetilde{A} \otimes_{\widetilde{R}} K 
= A \otimes_R K = K^{[1]}$.

($3^{\prime}$) $\widetilde{A} \otimes_{\widetilde{R}} 
k(\widetilde{P})$ is an $\mathbb{A}^1$-form over $k(\widetilde{P})$ 
for each height one prime ideal $\widetilde{P}$ of 
$\widetilde{R}$ (since, for any height one prime ideal $\widetilde{P}$ 
of $\widetilde{R}$, $\widetilde{P} \cap R \neq \m$).

\medskip

By Theorem \ref{Th2}, $\widetilde{A} = \widetilde{R}^{[1]}$; 
in particular, $\widetilde{A} \otimes_{\widetilde{R}} k(\widetilde{\m}) 
= k(\widetilde{\m})^{[1]}$. This shows that 
$A \otimes_R k(\m)$ is an $\mathbb{A}^1$-form over $k(\m)$ 
and hence $A \otimes_R k(P)$ is an $\mathbb{A}^1$-form over $k(P)$ 
for every prime ideal $P$ of $R$.

\bigskip

(II):
We now show that $A$ is finitely generated over $R$. 
By Lemma \ref{B-D_lem_2}, it is enough to take $R$ to be an 
integral domain; by Theorem \ref{O_f.g.} and Lemma \ref{L3}, 
it is enough to assume $R$ to be local and, by Lemma \ref{L4}, 
it is enough to take $R$ to be complete. Thus we assume that 
$R$ is a Noetherian local complete integral domain. 
Let $\widetilde{R}$ be the normalisation of $R$. 
Then the proof of (I) shows that 
$A \otimes_R \widetilde{R} = {\widetilde{R}}^{[1]}$; 
in particular, $A \otimes_R \widetilde{R}$ is finitely generated 
over $\widetilde{R}$. Since $\widetilde{R}$ is a finite module over $R$, 
by Lemma \ref{B-D_lem_1}, $A$ is finitely generated over $R$. 

\bigskip

(III):
Now $R$ is given to be an integral domain. By (I),
$A \otimes_R k(P)$ is an $\mathbb{A}^1$-form over $k(P)$ 
for every prime ideal $P$ of $R$.

\medskip

Let $\widetilde{R}$ be the normalisation of $R$. 
Then $\widetilde{R}$ is a Krull domain (\cite{Mat_RING}, p 91). 
Let $\widetilde{A} = A \otimes_R \widetilde{R}$. 
As before, there is a retraction 
$\widetilde{\Phi}: \widetilde{A} \onto \widetilde{R}$ 
with finitely generated kernel. 
We now have the following:

\medskip

$\widetilde{R}$ is a Krull domain with quotient field $K$ and 
$\widetilde{A}$ is a faithfully flat $\widetilde{R}$-algebra such that

\smallskip

($1^{\prime \prime}$) There is a retraction 
$\widetilde{\Phi} : \widetilde{A} \onto \widetilde{R}$ 
with finitely generated kernel.

($2^{\prime \prime}$) $\widetilde{A} \otimes_{\widetilde{R}} K = K^{[1]}$.

($3^{\prime \prime}$) $\widetilde{A} \otimes_{\widetilde{R}} 
k(\widetilde{P})$ is an $\mathbb{A}^1$-form over $k(\widetilde{P})$ 
for each prime ideal $\widetilde{P}$ of $\widetilde{R}$ 
(since $k(\widetilde{P})$ is algebraic over $k(\widetilde{P} \cap R)$).

\medskip

Using Theorem \ref{Th2}, we get that 
$A \otimes_R \widetilde{R} = \widetilde{R}[\widetilde{I}T]$ for some 
invertible ideal $\widetilde{I}$ of $\widetilde{R}$. 
Let $\widetilde{I} = (a_1, a_2, \dots , a_n) \widetilde{R}$ and let 
$\alpha_1, \dots, \alpha_n \in {\widetilde{I}}^{-1}$ be such that
$a_1 \alpha_1 + \dots a_n \alpha_n =1$. Set
 $b_{ij}:=a_i \alpha_j (\in \widetilde{R})$, $1 \le i,j \le n$.
Let $a_p T = \underset{q=1} {\overset{s_p} {\sum}} {u_{pq} \otimes c_{pq}}$ 
where $c_{pq} \in \widetilde{R}$ and $u_{pq} \in A$.

\medskip

By (II), $A$ is finitely generated; let $A = R[y_1, y_2, \dots, y_t]$ 
where each $y_{\ell} \in Ker \ \Phi$. Then 
\[
y_{\ell} \otimes 1 = \underset{m=0} {\overset{r_{\ell}} {\sum}} 
~ \underset{m_1 +m_2 + \dots + m_n =m}
  \sum {d_{\ell~ m_1 m_2 \dots m_n}~ {a_1}^{m_1}  
	{a_2}^{m_2} \dots {a_n}^{m_n}} ~ {T^m} 
\] 
 for some $d_{\ell ~ m_1 m_2 \dots m_n} \in \widetilde{R}$. 

\medskip

Now, let $R'$ be the $R$-subalgebra of $\widetilde{R}$ 
generated by the elements $a_1, a_2, \dots , a_n$; 
$b_{ij}$ where $1 \le i,j \le n$; 
$c_{pq}$ where $1 \le q \le s_p$, $1 \le p \le n$; 
and $d_{\ell ~ m_1 m_2 \dots m_n}$ 
where $m_1 +m_2 + \dots +m_n =m$, 
$0 \le m \le r_{\ell}$, $1 \le \ell \le t$. 
Let $I$ be the ideal $(a_1, a_2, \dots , a_n)R'$. 
Then $R'$ is a finite birational extension of $R$ and 
$I$ is an invertible ideal of $R'$.

\medskip

Since $A$ is faithfully flat over $R$, we have 
$A \otimes_R R' \subseteq A \otimes_R \widetilde{R} 
\subseteq A \otimes_R K = K[T]$. 
Now considering $A \otimes_R R'$ and $R'[IT]$ as subrings of 
$A \otimes_R K$, it is easy to see that $A \otimes_R R' = {R'}[IT]$.

This completes the proof.
\end{proof}

\begin{rem} \label{Rem_2}
The above proof shows that in the statement of Proposition \ref{Th3}, 
it is enough to assume in (2) that the generic fibres are 
$\mathbb{A}^1$-forms. 
(In the proof take $\widetilde{R}$ to be the integral closure of $R$ in 
$L$ where $L$ is a finite extension of $K$ such that 
$A \otimes_R L = L^{[1]}$.)
\end{rem}

We now prove Theorem A.

\begin{thm} \label{Th4}
Let $ \mathbb{Q} \hookrightarrow R$ be a Noetherian ring and 
$A$ be a flat $R$-algebra with a retraction $\Phi : A \onto R$ such that

\begin{enumerate}
 \item [\rm(1)]Ker $\Phi$ is finitely generated.
 \item [\rm(2)] $A \otimes_R k(P) = k(P)^{[1]}$ at each minimal 
 prime ideal $P$ of $R$.
 \item [\rm(3)] $A \otimes_R k(P)$ is an integral domain 
 at each height one prime ideal $P$ of $R$.
\end{enumerate}

Then:

\begin{enumerate}
 \item [\rm(I)] $A$ is an $\mathbb{A}^1$-fibration over $R$.
 \item [\rm(II)] If $R$ is an integral domain, then there exists a 
 finite birational extension $R'$ of $R$ and an invertible ideal 
 $I$ of $R'$ such that $A \otimes_R R' \cong Sym_{R'} ( I )$.
 \item [\rm(III)] If $R_{red}$ is seminormal, then $A \cong Sym_R (I)$ 
 for some finitely generated rank one projective $R$-module $I$.
\end{enumerate}
\end{thm}

\begin{proof}

(I):
By Proposition \ref{Th3}, it is enough to show that 
$A \otimes_R k(P) = k(P)^{[1]}$ 
for each prime ideal $P$ in $R$ of height one.

\medskip

Fix a prime ideal $P$ in $R$ of height one. Replacing $R$ by $R_P$, 
we assume that $R$ is a one-dimensional Noetherian local ring with 
maximal ideal $\m$ and residue field $k$. 
Moreover, replacing $R$ by $R/P_0$ for some minimal prime ideal $P_0$, 
we may further assume that $R$ is an integral domain with quotient field $K$. 
We show that $A \otimes_R k = k^{[1]}$.

\medskip

Note that $k$ is a field of characteristic $0$. 
By Krull-Akizuki theorem, there exists a discrete valuation ring
 $(\widetilde{R}, \pi, \widetilde{k})$ such that 
$R \subset \widetilde{R} \subset K$ 
and $\widetilde{k}$ is a finite separable extension of $k$. 
Let $\widetilde{A} = A \otimes_R \widetilde{R}$. 
Since separable $\mathbb{A}^1$-forms are $\mathbb{A}^1$, 
to show that $A \otimes_R k = k^{[1]}$, it is enough to show that 
$\widetilde{A} /\pi \widetilde{A} 
(= A \otimes_R \widetilde{k})= \widetilde{k}^{[1]}$ and 
hence enough to show that $\widetilde{A} = \widetilde{R}^{[1]}$.

\medskip

Now, the retraction $\Phi : A \onto R$ with finitely generated kernel 
induces a retraction $\widetilde{\Phi}: \widetilde{A} \onto 
\widetilde{R}$ with finitely generated kernel. 
Also $\widetilde{A}[1/\pi] = K^{[1]}$. 
Using Lemma \ref{L2}, 
we get $x \in Ker \ \widetilde{\Phi} \backslash \pi \widetilde{A}$ 
such that $\widetilde{A}[1/\pi] = K[x]$.

\medskip

Let $B = \widetilde{R}[x] \subset \widetilde{A}$. 
We will show that $\widetilde{A} = B$. 
Since $\pi$ is a non-zero divisor and since 
$ \widetilde{A}_{\pi} = B_{\pi}$, by Lemma \ref{L0a}, 
it suffices to show that $\pi \widetilde{A} \cap B = \pi B$.

\medskip

Let $D = A \otimes_ R k$. Then 
$ \widetilde{A}/ \pi \widetilde{A} = \widetilde{A} \otimes_{\widetilde{R}} 
\widetilde{k} = (A \otimes_R k) \otimes_k \widetilde{k} = 
D \otimes_k \widetilde{k}$. By hypothesis, $D$ is an integral domain and 
hence, as $\widetilde{k}|_k$ is separable, 
$\widetilde{A}/\pi \widetilde{A} = D \otimes_k \widetilde{k}$ is a 
reduced ring. Note that $\widetilde{A}/ \pi \widetilde{A}$ is a finite 
flat module over $D$ and hence $\widetilde{A}$ has only 
finitely many minimal prime ideals $P_1, P_2, \dots , P_n$ 
containing $\pi \widetilde{A}$. 
To show that $\pi \widetilde{A} \cap B = \pi B$, 
it is enough to show that $P_i \cap B = \pi B$ for some $i$.

\medskip

Suppose, if possible, that $P_i \cap B \ne \pi B$ for all $i$. 
Let $P_i \cap B = Q_i$. Then $ht \ Q_i > 1$, i.e., $Q_i$s 
are maximal ideals of $B$ (since dim $B$ $= 2$). 
Let $t$ be the number of distinct ideals in the family 
$\{ Q_1, Q_2, \dots, Q_n \}$. By reindexing, if necessary, 
we assume that $Q_1, Q_2, \dots, Q_t$ are all distinct. 
Let $I_i = \underset{P_{j} \cap B = Q_i} {\overset{} {\cap}} P_{j}$. 
Since $Q_i$s are pairwise comaximal, $I_i$s are pairwise comaximal. 
Thus $\widetilde{A}/ \pi \widetilde{A} = 
\widetilde{A}/ I_1 \times \widetilde{A}/ I_2 \times \dots \times 
\widetilde{A}/ I_t$.

Since $D = A \otimes_R k$, the retraction $\Phi: A \onto R$ 
induces a retraction $\Phi_k : D \onto k$. 
Let $\m_0$ be a maximal ideal of $D$ such that $D/\m_0 = k$. 
Note that $D \hookrightarrow D_{\m_0}$ and hence, due to flatness, 
$D \otimes_k \widetilde{k} \hookrightarrow D_{\m_0} \otimes_k \widetilde{k}$. 
Since $D_{\m_0}$ is local and since $\widetilde{k}|_k$ is a finite 
extension, $D_{\m_0} \otimes_k \widetilde{k}$ is also local with 
maximal ideal $\m_0 (D_{\m_0} \otimes_k \widetilde{k})$ and 
residue field $\widetilde{k}$. As the local ring 
$D_{\m_0} \otimes_k \widetilde{k}$ is a localisation of 
$D \otimes_k \widetilde{k} = \widetilde{A}/ \pi \widetilde{A}$, 
it follows that there exists a prime ideal $\p$ of 
$\widetilde{A}/ \pi \widetilde{A}$ such that 
$D_{\m_0} \otimes_k \widetilde{k} = (\widetilde{A}/ \pi \widetilde{A})_{\p}$.

\smallskip

Note that $\widetilde{A}/ \pi \widetilde{A} = 
D \otimes_k \widetilde{k} \hookrightarrow D_{\m_0} \otimes_k 
\widetilde{k} = (\widetilde{A}/ \pi \widetilde{A})_{\p}$. 
As the map $\widetilde{A}/ \pi \widetilde{A} 
\longrightarrow (\widetilde{A}/ \pi \widetilde{A})_{\p}$ is one-one, 
it follows that the zero divisors of 
$\widetilde{A}/ \pi \widetilde{A}$ are contained in $\p$. 
Consequently, $\overline{P_i} \subset \p$ where $\overline{P_i}$ is 
the image of $P_i$ in $\widetilde{A}/ \pi \widetilde{A}$. 
But this would imply that the local ring 
$(\widetilde{A}/ \pi \widetilde{A})_{\p}$ is a product of 
$t$ rings which is possible only if $t=1$. 
So $P_i \cap B = Q$ for all $i$, which implies that 
$\pi \widetilde{A} \cap B = Q$. Note that the retraction 
$\widetilde{\Phi} : \widetilde{A} \onto \widetilde{R}$ induces a 
retraction $\widetilde{\Phi}_{\pi}: \widetilde{A}/ \pi 
\widetilde{A} \onto \widetilde{R}/ \pi \widetilde{R}$. 
Now since $\pi \widetilde{A} \cap B = Q$, the retraction 
$\widetilde{\Phi}_{\pi}$ induces a retraction 
$\widetilde{\Phi}_{\pi}^{\prime} : B/Q \onto \widetilde{k}$. 
But $Q$ is a maximal ideal of $B$, i.e., $B/Q$ is a field. 
Hence $\widetilde{\Phi}_{\pi}^{\prime}$ is an isomorphism. 
As $x \in Ker \ \widetilde{\Phi}$, it then follows that 
$x \in Q \subset \pi \widetilde{A}$ and hence 
$x \in \pi \widetilde{A}$, a contradiction.

\smallskip 

Thus $\pi \widetilde{A} \cap \widetilde{R}[x] = \pi \widetilde{R}[x]$ and 
hence $\widetilde{A} = \widetilde{R}^{[1]}$ showing that 
$A \otimes_R k = k^{[1]}$.

\bigskip

(II): Follows from (III) of Proposition \ref{Th3}.

\bigskip

(III): Follows from (I) and the result (\cite{Asanuma_fibre_ring}, 3.4) of 
Asanuma, using results of Hamann (\cite{Haman_Invariance}, 2.6 or 2.8) 
and Swan (\cite{On_Sem}, 6.1); also see (\cite{Greither_Seminormal}).
\end{proof}

\begin{rem}
Examples from existing literature would show that the hypotheses in our 
results cannot be relaxed. For instance, the hypothesis that 
``$Ker \ \Phi$ is finitely generated'' is necessary in all the results 
as can be seen from the example: 
Let $(R, \pi)$ be a DVR and $A=R[X, X/\pi, X/\pi^2, \dots, X/\pi^n, \dots]$.

\medskip

An example of Eakin-Silver (\cite{E-S}, 3.15) shows that the hypothesis 
``$A$ has a retraction to $R$'' is necessary in Proposition \ref{Th0}. 
Even if $R$ is local and factorial and $A$ Noetherian, 
the hypothesis ``$A$ has a retraction to $R$'' would still be 
necessary in Theorem \ref{Th2} even to conclude that 
$A$ is finitely generated as has been shown recently in 
\cite{LOC_A1_CODIM1}. Even if $A$ is finitely generated, 
the hypothesis ``$A$ has retraction to $R$'' would still be necessary 
in Theorem \ref{Th2} to conclude that $A$ is a 
symmetric algebra (consider $R=k[[t_1, t_2]]$ where $k$ is any field and 
$A=R[X,Y]/(t_1 X + t_2 Y -1)$).

\medskip

The following example of Yanik (\cite{Yanik}, 4.1) shows the necessity 
of seminormality hypothesis in the passage from (I) to (III) 
in Theorem \ref{Th4}: Let $k$ be a field of characteristic zero, 
$R=k[[t^2, t^3]]$ and $A=R[X, tX^2] + (t^2, t^3)R[X]$; 
also see \cite{Greither_Seminormal}.

\medskip

For other examples (e.g.,  the necessity of 
``geometrically integral'' in Proposition \ref{Th3}, 
the necessity of ``$\mathbb Q \hookrightarrow R$'' in Theorem \ref{Th4} and 
the necessity of ``flatness''), see \cite{BD_A1FIBSUBALG}, Section 4.
\end{rem}

{\noindent \bf Acknowledgement.} The authors thank N. Onoda 
for pointing out a gap in argument in an earlier draft of the 
paper and to S.M. Bhatwadekar and Neena Gupta for useful suggestions.
The first author also thanks the National Board for Higher Mathematics,
India, for financial support.

\bibliography{reference1}

\providecommand{\bysame}{\leavevmode\hbox to3em{\hrulefill}\thinspace}
\providecommand{\MR}{\relax\ifhmode\unskip\space\fi MR }
\providecommand{\MRhref}[2]{%
  \href{http://www.ams.org/mathscinet-getitem?mr=#1}{#2}
}
\providecommand{\href}[2]{#2}
\begin{thebibliography}{Ham75}

\bibitem[Asa87]{Asanuma_fibre_ring}
Teruo Asanuma, \emph{Polynomial fibre rings of algebras over {N}oetherian
  rings}, Invent. Math. \textbf{87} (1987), no.~1, 101--127.

\bibitem[BD95]{BD_A1FIBSUBALG}
S.~M. Bhatwadekar and Amartya~K. Dutta, \emph{On {$\mathbb{A}^1$}-fibrations of
  subalgebras of polynomial algebras}, Compositio Math. \textbf{95} (1995),
  no.~3, 263--285.

\bibitem[BDO]{LOC_A1_CODIM1}
S.~M. Bhatwadekar, Amartya~K. Dutta, and Nobuharu Onoda, \emph{On algebras
  which are locally {$\mathbb{A}^1$} in codimension-one}, Preprint.

\bibitem[DO07]{DO_CODIM1}
Amartya~K. Dutta and Nobuharu Onoda, \emph{Some results on codimension-one
  {$\mathbb{A}^1$}-fibrations}, J. Algebra \textbf{313} (2007), no.~2,
  905--921.

\bibitem[Dut95]{D_MOR}
Amartya~K. Dutta, \emph{On {$\mathbb{A}^1$}-bundles of affine morphisms}, J.
  Math. Kyoto Univ. \textbf{35} (1995), no.~3, 377--385.

\bibitem[ES72]{E-S}
Paul Eakin and James Silver, \emph{Rings which are almost polynomial rings},
  Trans. Amer. Math. Soc. \textbf{174} (1972), 425--449.

\bibitem[Gre86]{Greither_Seminormal}
Cornelius Greither, \emph{A note on seminormal rings and {${\bf A}\sp
  1$}-fibrations}, J. Algebra \textbf{99} (1986), no.~2, 304--309.

\bibitem[Ham75]{Haman_Invariance}
Eloise Hamann, \emph{On the {$R$}-invariance of {$R[x]$}}, J. Algebra
  \textbf{35} (1975), 1--16.

\bibitem[Mat89]{Mat_RING}
Hideyuki Matsumura, \emph{Commutative ring theory}, second ed., Cambridge
  Studies in Advanced Mathematics, vol.~8, Cambridge University Press,
  Cambridge, 1989, Translated from the Japanese by M. Reid.

\bibitem[Ono84]{O_Subring}
Nobuharu Onoda, \emph{Subrings of finitely generated rings over a
  pseudogeometric ring}, Japan. J. Math. (N.S.) \textbf{10} (1984), no.~1,
  29--53.

\bibitem[RS79]{RS_FIND}
Peter Russell and Avinash Sathaye, \emph{On finding and cancelling variables in
  {$k[X,\,Y,\,Z]$}}, J. Algebra \textbf{57} (1979), no.~1, 151--166.

\bibitem[Swa80]{On_Sem}
Richard~G. Swan, \emph{On seminormality}, J. Algebra \textbf{67} (1980), no.~1,
  210--229.

\bibitem[Yan81]{Yanik}
Joe Yanik, \emph{Projective algebras}, J. Pure Appl. Algebra \textbf{21}
  (1981), no.~3, 339--358.

\end{thebibliography}
\bibliographystyle{amsalpha}

\end{document}